\documentclass[11pt]{amsart}
\usepackage[all]{xy}
\usepackage[parfill]{parskip}

\usepackage{verbatim}
\usepackage{color}

\usepackage{amsmath, amscd, graphicx, latexsym, hyperref, times}
\usepackage{esint} 
\usepackage{graphicx}
\usepackage[abs]{overpic}
\usepackage{tikz}
\usepackage{tikz-cd}
\usetikzlibrary{arrows, patterns}

\oddsidemargin .25in \theoremstyle{plain}
\newtheorem{theorem}{Theorem}

\newtheorem{lemma}{Lemma}
\newtheorem{proposition}{Proposition}
\newtheorem{corollary}{Corollary}

\newtheorem{remark}{Remark}

\numberwithin{equation}{section}
\numberwithin{lemma}{section}
\numberwithin{proposition}{section}
\numberwithin{corollary}{section}
\numberwithin{remark}{section}
\allowdisplaybreaks[3]

\newcommand{\dd}{\mathop{}\!\mathrm{d}}
\usepackage{cleveref}
\crefname{theorem}{Theorem}{Theorem}
\crefname{lemma}{Lemma}{Lemma}
\crefname{proposition}{Proposition}{Proposition}
\crefname{conjecture}{conjecture}{conjecture}
\crefname{remark}{Remark}{Remark}
\crefname{equation}{}{}
\crefname{section}{Section}{Section}
\crefname{figure}{Figure}{Figure}
\usepackage{appendix}
\begin{document}

\title[Modified Bakry-\'Emery $\Gamma_2$ criterion and Tsallis entropy]{A modified Bakry-\'Emery $\Gamma_2$ criterion inequality and the monotonicity of the Tsallis entropy}
\author{Xiaohan Cai}
\address{School of Mathematical Sciences, Shanghai Jiao Tong University, Shanghai 200240, China}
\email{xiaohancai@sjtu.edu.cn}
\author{Xiaodong Wang}
\address{Department of Mathematics,
Michigan State University, East Lansing, MI 48824}
\email{xwang@math.msu.edu}

\thanks{}
\date{}
\begin{abstract}
The Bakry-\'Emery $\Gamma_2$ criterion inequality provides a method for establishing the logarithmic Sobolev inequality. We prove a one-parameter family of weighted Bakry-\'Emery $\Gamma_2$ criterion inequalities which in the limit case yields the improved constant due to Ji \cite{Ji24}.  Furthermore, we establish a modified weighted $\Gamma_2$ criterion inequality which could be interpreted as a monotonicity of the Tsallis entropy under the heat flow and yields a family of sharp Sobolev inequalities.



\end{abstract}
\maketitle
\section{Introduction}
	The sharp constants of Sobolev inequalities on Riemannian manifolds have been intensively studied from various perspectives over the last several decades.
    Among these inequalities, the logarithmic Sobolev inequality, introduced by Gross \cite{Gro75}, plays a vital role at the interface of analysis, geometry and probability. It is impossible to mention all important work in this field. We recommend \cite{Rot81,Bec99,BL00,Per02,Led06,Bak06,BE06,Gro06, Bre22} and the references therein.

    The logarithmic Sobolev inequality on compact Riemannian manifolds states as:
    \begin{align}\label{ineq. log Sobolev ineq}
        \fint_M |\nabla u|^2\geq
        \frac{C}{2}\left(
        \fint_M u^2\log u^2
        -\left(\fint_M u^2\right)
        \log\left(\fint_M u^2\right)
        \right),\quad \forall\, 0<u\in C^{\infty}(M).
    \end{align}

    The optimal constant of \cref{ineq. log Sobolev ineq}, denoted by $\alpha_M$, refers to the largest constant $C$ such that the inequality \cref{ineq. log Sobolev ineq} holds.
     Historically, Bakry and \'{E}mery \cite{BE06} proved that for $(M^n,g)$ with $Ric_g\geq n-1$, 
     \begin{align*}
         \alpha_M\geq  n.
     \end{align*}
     Later, Rothaus \cite[Theorem 1]{Rot86} obtained an improved lower bound for $\alpha_M$ under the same curvature hypothesis:
     \begin{align}\label{ineq. Rothaus}
         \alpha_M \geq 
         \frac{4n}{(n+1)^2}\lambda_1
         +\frac{(n-1)^2}{(n+1)^2}n
         ,
     \end{align}
     where $\lambda_1$ is the first eigenvalue of $-\Delta$ on $(M^n,g)$. The proof  consists of studying a critical equation of \cref{ineq. log Sobolev ineq} and a delicate choice of test functions.
     
    
    Another interesting way toward \cref{ineq. log Sobolev ineq} is the  Bakry-\'{E}mery $\Gamma_2$ criterion\footnote{The terminology $\Gamma_2$ comes from the Carr\'e du champ calculus systematically studied by Bakry and his collaborators, see \cite{Led00}. For the Laplace operator on a manifold $(M^n,g)$, $\Gamma_2(f,f)=|\nabla^2 f|^2+Ric(\nabla f,\nabla f)$.}. 
    Consider the  following inequality
    \begin{align}\label{ineq. Gamma2 criterion}
        \int_M f(|\nabla^2\log f|^2+Ric(\nabla \log f,\nabla \log f))
        \geq 
        C\int_M f|\nabla \log f|^2,\quad 
        \forall\, 0<f\in C^{\infty}(M).
    \end{align}
    The Bakry-\'{E}mery $\Gamma_2$ criterion \cite{BE06} states that if \cref{ineq. Gamma2 criterion} holds with $C=\Lambda$, then the following also holds with $C=\Lambda$:
    \begin{align}\label{ineq. equivalent log Sob}
        \fint_M f|\nabla\log f|^2
        \geq 
        2C\left(
        \fint_M f\log f
        -\left(\fint_M f\right)
        \log\left(\fint_M f\right)
        \right),\quad 
        \forall \, 0<f\in C^{\infty}(M),
    \end{align}
     which    
    is equivalent to, by substituting $u=\sqrt{f}$, the logarithmic Sobolev inequality \cref{ineq. log Sobolev ineq}:
    \begin{align*}
        \fint_M |\nabla u|^2\geq
        \frac{C}{2}\left(
        \fint_M u^2\log u^2
        -\left(\fint_M u^2\right)
        \log\left(\fint_M u^2\right)
        \right),\quad \forall\, 0<u\in C^{\infty}(M).
    \end{align*}
    In other words, if we let $\Lambda_M$ be the largest constant such that \cref{ineq. Gamma2 criterion} holds with $C=\Lambda_M$, then the criterion yields
    \begin{align}\label{ineq. Bakry Emery criterion}
        \alpha_M\geq \Lambda_M.
    \end{align}
    To get their $\Gamma_2$ criterion, Bakry-\'Emery consider how the Shannon entropy of a positive function $f$:
    \begin{align}\label{eq. Shannon entropy}
        S(f):=-\int_Mf\log f
    \end{align}
    evolve under the heat flow $\partial_t f=\Delta f$. Then 
    \cref{ineq. Gamma2 criterion} could be interpreted as an inequality involving the first and second order derivatives of the Shannon entropy \cref{eq. Shannon entropy}. 
    After integrating \cref{ineq. Gamma2 criterion} over $t\in[0,+\infty)$, one gets \cref{ineq. equivalent log Sob}.  We shall present a  detailed illustration at the end of \cref{sec. proof of thm}.
    
    
    Based on the Bakry-\'Emery criterion, a lower bound of $\Lambda_M$ is of independent interest. Bakry and \'{E}mery \cite{BE06} actually proved the stronger result that $\Lambda_M\geq  n$ if $(M^n,g)$ satisfies $Ric_g\geq n-1$. This is optimal as $\Lambda_{\mathbb{S}^n}=  n$.

     In the recent seminal work of Guillen-Sylvestre \cite{GS25} proving the existence of global smooth solutions for the Landau equation with Coulomb potential,  the inequality \cref{ineq. Gamma2 criterion} with an improved constant for even functions on $\mathbb{S}^3$ plays a crucial role. Afterward, Ji \cite[Theorem 2.3]{Ji24} proved the following general result.
    \begin{theorem}[Ji, \cite{Ji24}]\label{thm. Ji}
        Let $(M^n,g)$ be a compact Riemannian manifold with $Ric\geq n-1$. Denote $\lambda_1$ as the first eigenvalue of $-\Delta $ on $(M^n,g)$. Then for any $0<f\in C^{\infty}(M)$,
        \begin{align*}
            \int_M f(|\nabla^2\log f|^2+Ric(\nabla \log f,\nabla \log f))
        \geq 
        \left(
        \frac{4n-1}{n(n+2)}\lambda_1
            +\frac{(n-1)^2}{n(n+2)}n
            \right)
            \int_M f|\nabla \log f|^2.
        \end{align*}
        That is,
        \begin{align*}
            \Lambda_M\geq \frac{4n-1}{n(n+2)}\lambda_1
            +\frac{(n-1)^2}{n(n+2)}n.
        \end{align*}
    \end{theorem}

    Our first result is a one-parameter family of weighted Bakry-\'{E}mery $\Gamma_2$ criterion inequality.
    
 \begin{theorem}\label{thm. CW}
     Let $(M^n,g)$ ($n\geq 2$) be a compact Riemannian manifold  with $Ric\geq n-1$.
     Denote $\lambda_1$ as the first eigenvalue of $-\Delta $ on $(M^n,g)$.
    Then for $s\in(-\infty,0]\cup\left[\frac{4(n+2)}{4n-1},+\infty\right)$ and $0<v\in C^{\infty}(M)$ , there holds
    \begin{align}\label{ineq. weighted BE Gamma2}
        \int_M v^s(|\nabla^2 v|^2+Ric(\nabla v,\nabla v))
        \geq 
        \left(\left(1-\frac{(n-1)^2s}{(n+2)(ns-4)}\right)\lambda_1
        +\frac{(n-1)^2ns}{(n+2)(ns-4)}
        \right)
        \int_M v^s|\nabla v|^2.
    \end{align}
 \end{theorem}
 As a limiting case, we recover \cref{thm. Ji} (see \cref{cor. recover Ji}). The result could be useful in other situations.  Our proof is direct and simple.
 

 In the second part of this paper, we investigate the monotonicity of a generalized Shannon entropy under the heat flow. In 1988, Tsallis introduced the following entropy \cite{Tsa88}: for $p\in(0,1)\cup(1,+\infty)$ and $0<u\in C^{\infty}(M)$,
\begin{align*}
    T_p(u):=\frac{\int_M u^p-\int_M u}{1-p}.
\end{align*}
It is straightforward to see that as $p\to 1$, Tsallis entropy converges to the Shannon entropy, making it a natural generalization of the latter. Moreover, the Tsallis entropy is also widely used in statistical mechanics, see for example \cite{PP93, Lut03, Tsa09}. We study explicitly how the Tsallis entropy evolves under the heat flow. Its second derivative is given by a modified weighted $\Gamma_2$ integral with an additional term involving $\Delta u$.
 We generalize the Bakry-\'Emery $\Gamma_2$ criterion inequality to this setting. 
 \begin{theorem}\label{thm. modified weighted Gamma2}
    Let $(M^n,g)$ ($n\geq 2$) be a compact Riemannian manifold with $Ric\geq n-1$. 
    Denote $\lambda_1$ as the first eigenvalue of $-\Delta $ on $(M^n,g)$.
    Then for $s\in\left(-\infty,-\frac{2(2n^2+1)}{4n-1}\right]
    \cup[2,+\infty)$ and $0<v\in C^{\infty}(M)$ , there holds
    \begin{align}
        &\int_M v^s(|\nabla^2 v|^2+Ric(\nabla v,\nabla v)-v^{-1}|\nabla v|^2\Delta v)\notag\\
        \geq &
        \left(\left(1-\frac{(n-1)^2(s-2)}{n(n+2)(s+2)}
        \right)\lambda_1
        +\frac{(n-1)^2(s-2)}{(n+2)(s+2)}
        \right)
        \int_M v^s|\nabla v|^2.\label{ineq. interpolated BE Gamma2}
    \end{align}
\end{theorem}
In terms of the Tsallis entropy under the heat flow, \cref{thm. modified weighted Gamma2} yields a differential inequality that relates its first derivative and its second derivative.
\begin{theorem}\label{thm. ODE of T entropy}
    Let $(M^n,g)$ ($n\geq 2$) be a compact Riemannian manifold with  $Ric\geq n-1$.
    Denote $\lambda_1$ as the first eigenvalue of $-\Delta $ on $(M^n,g)$.
    Then for $p\in\left[\frac{2(n-1)^2}{2n^2+1},1\right)\cup(1,2]$ and under the heat flow $\partial_t u=\Delta u$ with a positive initial data,  we have:
    \begin{align}\label{ineq. in thm ODE of T entropy}
        \frac{\dd^2}{\dd t^2}T_p(u)
        +2\left(\left(
        1-\frac{(n-1)^2}{n(n+2)}\left(\frac{2}{p}-1\right)
        \right)\lambda_1
        +\frac{(n-1)^2}{n+2}\left(\frac{2}{p}-1\right)
        \right)
        \frac{\dd}{\dd t}T_p(u)
        \leq 0.
    \end{align}
\end{theorem}
As the Bakry-\'Emery criterion did for the logarithmic Sobolev inequality, \cref{thm. ODE of T entropy} implies a family of Sobolev inequalities, which recovers the logarithmic Sobolev inequality as $q\to 2$. 
\begin{corollary}\label{cor. some Sobolev inequalities}
Let $(M^n,g)$ ($n\geq 2$) be a compact manifold with $Ric\geq n-1$. 
Denote $\lambda_1$ as the first eigenvalue of $-\Delta $ on $(M^n,g)$.
Then for $q\in[1,2)\cup\left(2,\frac{2n^2+1}{(n-1)^2}\right]$, we have
    \begin{align*}
        \fint_M|\nabla v|^2
        \geq
        \frac{1}{q-2}
        \left(\left(
        1-\frac{(n-1)^2}{n(n+2)}(q-1)
        \right)\lambda_1
        +\frac{(n-1)^2}{n+2}(q-1)
        \right)
        \left(
        \left(\fint_M v^q\right)^{\frac{2}{q}}
        -\fint_M v^2
        \right).
    \end{align*}
\end{corollary}
\begin{remark}\label{rmk. limitation of q}
    The limitation of the range $q\leq \frac{2n^2+1}{(n-1)^2}$ is mysterious and is due to the heat equation to some extent, as observed by Bakry-\'Emery \cite{BE85}.
\end{remark}
 This family of Sobolev inequalities is in fact not new. Historically, Bakry-\'Emery \cite{BE85} proved a sharp Sobolev inequality for $q\in(2,\frac{2n^2+1}{(n-1)^2}]$. Bidaut-Veron-Veron \cite[Theorem 6.1]{BVV91} extended this to $q\in(2,\frac{2n}{n-2}]$.
After that, Fontenas \cite{Fon97} improved the sharp Sobolev constant for $q\in(2,\frac{2n}{n-2}]$ and Dolbeault-Esteban-Loss \cite[Theorem 4]{DEL14} extended this to $q\in [1,2)$. Their results could be summarized as follows.
\begin{theorem}[\cite{Fon97, DEL14}]\label{thm. DEL14}
    Let $(M^n,g) (n\geq 3)$ be a compact manifold with $Ric\geq n-1$.
    Denote $\lambda_1$ as the first eigenvalue of $-\Delta $ on $(M^n,g)$.
    Then for $q\in[1,2)\cup (2,\frac{2n}{n-2}]$ and $u\in C^{\infty}(M)$, we have
    \begin{align*}
    \fint_{M}|\nabla u|^2
    \geq 
    \frac{1}{q-2}
    \left(\left(1-\frac{(n-1)^2(q-1)}{(q-2)+(n+1)^2}\right)\lambda_1
    +\frac{(n-1)^2(q-1)}{(q-2)+(n+1)^2} n \right)\times\\
    \left(
    \left(\fint_{M}|u|^{q}\right)^{\frac{2}{q}}
    -\fint_{M}|\nabla u|^2
    \right).
    \end{align*}
\end{theorem}

Our contribution here is to formulate a unified framework for a family of sharp Sobolev inequalities, including the Poincar\'e inequality (i.e. $q=1$) and the logarithmic Sobolev inequality (i.e. $q=2$), by deriving them from a monotonicity property of the Tsallis entropy under the heat flow, which is a natural generalization of the Bakry-\'Emery $\Gamma_2$ criterion.

Finally, we mention that the R\'enyi entropy \cite{Ren61} provides another generalization of the Shannon entropy, preserving most of its essential properties in the Euclidean setting. The R\'enyi entropy has garnered lots of study on both Euclidean space and noncompact manifolds, as seen in works such as \cite{Raj12,ST14, Kel16, Bra23, Sug25}.
Our work shows that the Tsallis entropy on compact manifolds with positive Ricci curvature has better properties.

    This paper is organized as follows. 
    In \cref{sec. preliminaries}, we present two useful integral identities. 
    In \cref{sec. proof of thm}, we prove our \cref{thm. CW} and present an alternative proof of \cref{thm. Ji}. 
    In \cref{sec. modified weighted Gamma2}, we prove \cref{thm. modified weighted Gamma2}.
    In \cref{sec. corollaries}, we investigate the monotonicity of the Tsallis entropy and prove \cref{thm. ODE of T entropy}, \cref{cor. some Sobolev inequalities}.
    
\textbf{Acknowledgement:} The first author's research is partially supported by NSFC-12031012, NSFC-12171313 and NSFC-12250710674.
\section{Some integral identities}\label{sec. preliminaries}
\begin{lemma}\label{lem.laplace squared}
        Let $(M^n,g)$ be a compact manifold without boundary. Then for any $u\in C^{\infty}(M)$, there holds
        \begin{align*}
            \int_M(\Delta u)^2
            =&\int_M
            \left(\left|\nabla^2 u\right|^2
            +\mathrm{Ric}(\nabla u,\nabla u)\right).
        \end{align*}
    \end{lemma}
    \begin{proof}
 By the Bochner formula,
        \begin{align*}
            \frac{1}{2}\Delta |\nabla u|^2
            =\left|\nabla^2 u\right|^2
            +\langle\nabla \Delta u,\nabla u\rangle
            +\mathrm{Ric}(\nabla u,\nabla u).
        \end{align*}
        Integrate it over $M$ and the proof finishes.
    \end{proof}

    \begin{lemma}\label{lem. integral identity}
        Let $(M^n,g)$ be a compact manifold without boundary. If $u\in C^{\infty}(M)$ is a positive function, then
        \begin{align*}
            \int_M u^{-1}|\nabla u|^2\Delta u
            =&\frac{n}{n+2}\int_M u^{-2}|\nabla u|^4
            -\frac{2n}{n+2}\int_M 
            \left\langle\nabla^2 u-\frac{\Delta u}{n}g,\frac{du\otimes du}{u}-\frac{1}{n}\frac{|\nabla u|^2}{u}g\right\rangle\\
            \int_M \left\langle
            \nabla^2u,\frac{\dd u\otimes \dd u}{u}
            \right\rangle
            =&\frac{1}{n+2}\int_M u^{-2}|\nabla u|^4
            +\frac{n}{n+2}\int_M \left\langle\nabla^2 u-\frac{\Delta u}{n}g,\frac{du\otimes du}{u}-\frac{1}{n}\frac{|\nabla u|^2}{u}g\right\rangle.
        \end{align*}
    \end{lemma}
    \begin{proof}
       By divergence theorem,
        \begin{align*}
            \int_M u^{-1}|\nabla u|^2\Delta u
            =&\int_M\mathrm{div}(u^{-1}|\nabla u|^2\nabla u)-(2u^{-1}\langle\nabla_{\nabla u}\nabla u,\nabla u\rangle-u^{-2}|\nabla u|^4)\\
            =&\int_M u^{-2}|\nabla u|^4
            -2\int_M 
            \left\langle
            \nabla ^2u
            ,\frac{du\otimes du}{u}
            \right\rangle\\
            =&\int_M u^{-2}|\nabla u|^4
            -\frac{2}{n}\int_M u^{-1}|\nabla u|^2\Delta u
            -2\int_M\left\langle
            \nabla ^2u-\frac{\Delta u}{n}g
            ,\frac{du\otimes du}{u}-\frac{1}{n}\frac{|\nabla u|^2}{u}g
            \right\rangle,
        \end{align*}
        where we used
        \begin{align}\label{eq. identity}
            \left\langle
            \nabla^2u,\frac{\dd u\otimes \dd u}{u}
            \right\rangle
            =\left\langle
            \nabla ^2u-\frac{\Delta u}{n}g
            ,\frac{du\otimes du}{u}-\frac{1}{n}\frac{|\nabla u|^2}{u}g
            \right\rangle
            +\frac{1}{n}u^{-1}|\nabla u|^2\Delta u.
        \end{align}
        Rearrange it and we get the first desired identity. The second one follows by using \cref{eq. identity} again.
    \end{proof}

 \section{Weighted Bakry-\'Emery $\Gamma_2$ criterion inequality}\label{sec. proof of thm}
 In this section, we shall prove \cref{thm. CW}, from which an alternative proof of \cref{thm. Ji} follows.
 \begin{proof}[Proof of \cref{thm. CW}:]
 For $s=0$, \cref{ineq. weighted BE Gamma2} follows from \cref{lem.laplace squared} and the fact that
 \begin{align*}
     \int_M (\Delta u)^2\geq 
     \lambda_1\int_M |\nabla u|^2,\quad \forall u\in C^{\infty}(M)
 \end{align*}
 
 For $s\in (-\infty,-2)\cup(-2,0)\cup\left[\frac{4(n+2)}{4n-1},+\infty\right)$, let $t:=s+2\in(-\infty,0)\cup(0,2)\cup\left[\frac{6(2n+1)}{4n-1},+\infty\right)$ and $u:=v^{\frac{t}{2}}$. It is straightforward to see that
 \begin{align*}
   v^{\frac{s}{2}} \nabla v
        =\frac{2}{t}\nabla u, \quad \text{and}\quad
        v^{\frac{s}{2}}\nabla^2v
        =\frac{2}{t}\nabla^2 u+
        \frac{2}{t}\left(\frac{2}{t}-1\right)
        u^{-1}\dd u\otimes \dd u.
 \end{align*}
 We compute 
 \begin{align*}
     & \left(\frac{t}{2}\right)^2
     \int_{M} v^s
     \left(|\nabla^2 v|^2
     +Ric(\nabla v,\nabla v)
     \right)\notag\\
     =&\int_M
     |\nabla^2 u|^2
     +\left(\frac{2}{t}-1\right)^2
     \int_M u^{-2}|\nabla u|^4
     +2\left(\frac{2}{t}-1\right)\int_M
     \left\langle\nabla^2 u,\frac{\dd u\otimes \dd u}{u}\right\rangle
     +\int_M Ric(\nabla u,\nabla u)\notag\\
     =&\alpha
     \int_M \left|\nabla^2 u\right|^2
     +(1-\alpha)\int_M \left|\nabla^2 u\right|^2
     +\int_M Ric(\nabla u,\nabla u)\notag
     \\
     &+\left(\frac{2}{t}-1\right)^2
     \int_M u^{-2}|\nabla u|^4
     +2\left(\frac{2}{t}-1\right)
     \int_M 
     \left\langle\nabla^2 u,
     \frac{\dd u\otimes \dd u}{u}\right\rangle\notag,\\
     \end{align*}
     where $\alpha$ is a positive parameter to be determined later. Replacing the 2nd term using \cref{lem.laplace squared}, we continue and use \cref{lem. integral identity}
     \begin{align}
      & \left(\frac{t}{2}\right)^2
     \int_{M} v^s
     \left(|\nabla^2 v|^2
     +Ric(\nabla v,\nabla v)
     \right)\notag\\=&\alpha
     \int_M \left(\left|\nabla^2 u-\frac{\Delta u}{n}g\right|^2+\frac{1}{n}(\Delta u)^2\right)
     +(1-\alpha)\int_M (\Delta u)^2
     +\alpha\int_M Ric(\nabla u,\nabla u)\notag
     \\
     &+\frac{2}{n+2}\left(\frac{2}{t}-1\right)
     \int_M u^{-2}|\nabla u|^4
     +\frac{2n}{n+2}\left(\frac{2}{t}-1\right)
     \int_M \left\langle\nabla^2 u-\frac{\Delta u}{n}g,\frac{\dd u\otimes \dd u}{u}-\frac{1}{n}\frac{|\nabla u|^2}{u}g\right\rangle\notag\\
     &+\left(\frac{2}{t}-1\right)^2
     \int_M u^{-2}|\nabla u|^4\notag\\
     =&\alpha\int_M \left|\nabla^2 u-\frac{\Delta u}{n}g\right|^2
     +\frac{2n}{n+2}\left(\frac{2}{t}-1\right)
     \int_M \left\langle\nabla^2 u-\frac{\Delta u}{n}g,\frac{\dd u\otimes \dd u}{u}-\frac{1}{n}\frac{|\nabla u|^2}{u}g\right\rangle\notag\\
     &+\left(\frac{2}{t}-1\right)
     \left(\frac{2}{t}-\frac{n}{n+2}\right)
     \int_M u^{-2}|\nabla u|^4\notag\\
     &+\left(1-\frac{n-1}{n}\alpha\right)\int_M (\Delta u)^2
     +\alpha\int_M Ric(\nabla u,\nabla u)\notag\\
     =&\int_M \left|\sqrt{\alpha}
     \left(\nabla^2 u-\frac{\Delta u}{n}g\right)
     +\frac{n}{n+2}
     \left(\frac{2}{t}-1\right)
     \frac{1}{\sqrt{\alpha}}
     \left(
     \frac{\dd u\otimes \dd u}{u}-\frac{1}{n}\frac{|\nabla u|^2}{u}g
     \right)
     \right|^2\notag\\
     &+\left(\frac{2}{t}-1\right)
     \left(
     \frac{2}{t}-\frac{n}{n+2}
     -\frac{n(n-1)}{(n+2)^2}
     \left(\frac{2}{t}-1\right)
     \frac{1}{\alpha}
     \right)\int_M u^{-2}|\nabla u|^4\notag\\
     &+\left(1-\frac{n-1}{n}\alpha\right)\int_M (\Delta u)^2
     +\alpha\int_M Ric(\nabla u,\nabla u),
     \label{eq. proof}
 \end{align}
 where we used the fact that 
 \begin{align*}
     \left|\frac{\dd u\otimes \dd u}{u}-\frac{1}{n}\frac{|\nabla u|^2}{u}g\right|^2=\frac{n-1}{n}u^{-2}|\nabla u|^4.
 \end{align*}
 Now we choose $\alpha$ such that the coefficient of $\int_M u^{-2}|\nabla u|^4$ in \cref{eq. proof} vanishes, i.e.
 \begin{align*}
     \alpha=\frac{n(n-1)\left(\frac{2}{t}-1\right)}{(n+2)\left(\frac{2}{t}(n+2)-n\right)}
        =\frac{(n-1)ns}{(n+2)(ns-4)}.
 \end{align*}
 Moreover, we want  the coefficients of the last two terms in \cref{eq. proof} to be nonnegative. Notice that
 \begin{align*}
     \alpha>0&\Longleftrightarrow
     s\in (-\infty,0)\cup \left(\frac{4}{n},+\infty\right)\\
     \alpha\leq \frac{n}{n-1}
     &\Longleftrightarrow
     s\in(-\infty,0]\cup\left[\frac{4(n+2)}{4n-1},+\infty\right).
 \end{align*}
 Therefore, with $s\in (-\infty,-2)\cup(-2,0)\cup\left[\frac{4(n+2)}{4n-1},+\infty\right)$, we could continue \cref{eq. proof} as follows:
 \begin{align*}
     & \left(\frac{t}{2}\right)^2
     \int_{M} v^s
     \left(|\nabla^2 v|^2
     +Ric(\nabla v,\nabla v)
     \right)\\
     =&\int_M \left|\sqrt{\alpha}
     \left(\nabla^2 u-\frac{\Delta u}{n}g\right)
     +\frac{n}{n+2}
     \left(\frac{2}{t}-1\right)
     \frac{1}{\sqrt{\alpha}}
     \left(
     \frac{\dd u\otimes \dd u}{u}-\frac{1}{n}\frac{|\nabla u|^2}{u}g
     \right)
     \right|^2\notag\\
     &+\left(\frac{2}{t}-1\right)
     \left(
     \frac{2}{t}-\frac{n}{n+2}
     -\frac{n(n-1)}{(n+2)^2}
     \left(\frac{2}{t}-1\right)
     \frac{1}{\alpha}
     \right)\int_M u^{-2}|\nabla u|^4\notag\\
     &+\left(1-\frac{n-1}{n}\alpha\right)\int_M (\Delta u)^2
     +\alpha\int_M Ric(\nabla u,\nabla u)\\
     \geq& 
     \left(1-\frac{n-1}{n}\alpha\right)\lambda_1
     \int_M |\nabla u|^2
     +(n-1)\alpha\int_M |\nabla u|^2\\
     =&\left(\frac{t}{2}\right)^2
     \left(\left(1-\frac{(n-1)^2s}{(n+2)(ns-4)}\right)\lambda_1
        +\frac{(n-1)^2ns}{(n+2)(ns-4)}
        \right)
        \int_M v^s|\nabla v|^2.
     \end{align*}

    This yields the desired inequality. For $s=-2$, the desired inequality follows by taking $s\to -2$ in \cref{ineq. weighted BE Gamma2}. 
 \end{proof}
 
    \begin{corollary}[Ji, \cite{Ji24}]\label{cor. recover Ji}
        Let $(M^n,g)$ be a compact Riemannian manifold with $Ric\geq n-1$. Denote $\lambda_1$ as the first eigenvalue of $-\Delta $ on $(M^n,g)$. Then for any $0<f\in C^{\infty}(M)$,
        \begin{align*}
        \int_M f(|\nabla^2\log f|^2+Ric(\nabla \log f,\nabla \log f))
        \geq 
        \left(
        \frac{4n-1}{n(n+2)}\lambda_1
            +\frac{(n-1)^2}{n(n+2)}n
        \right)
            \int_M f|\nabla \log f|^2.
        \end{align*}
\end{corollary}
\begin{proof}
    For any fixed $0<f\in C^{\infty}(M)$, let $v:=-\log f$. Then for $s\ll -1$, the function $1-\frac{v}{s}$ is positive on $M$. Then we could apply \cref{thm. CW} to $1-\frac{v}{s}$ and get
    \begin{align*}
        &\int_M \left(1-\frac{v}{s}\right)^s\frac{1}{s^2}
        \left(|\nabla^2 v|^2+Ric(\nabla v,\nabla v)\right)\\
        \geq &
        \left(\left(1-\frac{(n-1)^2s}{(n+2)(ns-4)}\right)\lambda_1
        +\frac{(n-1)^2ns}{(n+2)(ns-4)}
        \right)
        \int_M \left(1-\frac{v}{s}\right)^s\frac{1}{s^2}|\nabla v|^2.
    \end{align*}
    Multiply both sides by $s^2$ and let $s$ go to $-\infty$, we further notice that $\lim_{s\to-\infty}(1-\frac{v}{s})^s=e^{-v}=f$ and $\lim_{s\to-\infty}\frac{(n-1)^2ns}{(n+2)(ns-4)}=\frac{(n-1)^2}{n+2}$, so
    \begin{align*}
        \int_M f(|\nabla^2\log f|^2+Ric(\nabla \log f,\nabla \log f))
        \geq 
        \left(\frac{4n-1}{n(n+2)}\lambda_1
            +\frac{(n-1)^2}{n(n+2)}n
            \right)
        \int_M f|\nabla \log f|^2.
    \end{align*}
    This finishes our proof.
\end{proof}

\begin{remark}
    Ji's original arguments utilize a linear combination of two integral inequalities  and involve two delicately chosen parameters $\theta$ and $\beta$. In contrast, in our proof of \cref{thm. CW} (and hence \cref{cor. recover Ji}), only one parameter $\alpha$ is used.
\end{remark}

Now we elucidate the Bakry-\'Emery $\Gamma_2$ criterion precisely.

Given $0<f_0\in C^{\infty}(M)$, consider the heat equation on $(M^n,g)$ with the initial data $f_0$, i.e.
    \begin{align*}
        \begin{cases}
            \partial_tf=\Delta f\\
            f(0,\cdot)=f_0.
        \end{cases}
    \end{align*}
    It is straightforward to see that the time derivatives of the Shannon entropy \cref{eq. Shannon entropy} are as follows:
    \begin{align*}
        \frac{\dd }{\dd t}S(f)
        =&\int_M f|\nabla\log f|^2,\\
        \frac{\dd^2}{\dd t^2}S(f)
        =&-2\int_M f\left(
        |\nabla^2\log f|^2+Ric(\nabla \log f,\nabla \log f)
        \right).
    \end{align*}
    Hence \cref{cor. recover Ji} could be interpreted as
    \begin{align*}
        \frac{\dd^2}{\dd t^2}S(f)
        +2\left(
        \frac{4n-1}{n(n+2)}\lambda_1
            +\frac{(n-1)^2}{n(n+2)}n
        \right)
        \frac{\dd}{\dd t}S(f)
        \leq 0.
    \end{align*}
    Integrating for $t$ over $[0,+\infty)$ and noticing that $\lim_{t\to+\infty}f(t,\cdot)=\fint_M f_0$, we get
    \begin{align*}
        0-\int_Mf_0|\nabla\log f_0|^2
        +2\left(
        \frac{4n-1}{n(n+2)}\lambda_1
            +\frac{(n-1)^2}{n(n+2)}n
        \right)
        \left(
        -\int_M f_0 \log\left(\fint_M f_0\right)
        +\int_M f_0\log f_0
        \right)
        \leq 0.
    \end{align*}
    Rearranging it gives the logarithmic Sobolev inequality \cref{ineq. equivalent log Sob} with $C
    =\frac{4n-1}{n(n+2)}\lambda_1
            +\frac{(n-1)^2}{n(n+2)}n
        $.
        \section{Modified weighted Bakry-\'Emery $\Gamma_2$ inequality}\label{sec. modified weighted Gamma2}
        In this section, we prove the modified weighted Bakry-\'Emery $\Gamma_2$ inequality in \cref{thm. modified weighted Gamma2}. The strategy is similar to that employed in the proof of \cref{thm. CW}.
        \begin{proof}[Proof of \cref{thm. modified weighted Gamma2}:]
     First, consider $s\in\left(-\infty,-\frac{2(2n^2+1)}{4n-1}\right]\cup(2,+\infty)$. Let $t:=s+2\in\left(-\infty,-\frac{4(n-1)^2}{4n-1}\right]\cup(4,+\infty)$ and $u:=v^{\frac{t}{2}}$. It is straightforward to see that
 \begin{align*}
        v^{\frac{s}{2}} \nabla v
        =\frac{2}{t}\nabla u,\quad \text{and}\quad
        v^{\frac{s}{2}}\nabla^2v
        =\frac{2}{t}\nabla^2 u+
        \frac{2}{t}\left(\frac{2}{t}-1\right)
        u^{-1}\dd u\otimes \dd u
        ,\quad 
        \\
        v^sv^{-1}|\nabla v|^2\Delta v
        =\left(\frac{2}{t}\right)^3\left(\frac{2}{t}-1\right)u^{-2}|\nabla u|^4
        +\left(\frac{2}{t}\right)^3u^{-1}|\nabla u|^2\Delta u.
 \end{align*}
 Let $\alpha $ be a positive number to be determined later. Using the same strategy as in the derivation of \cref{eq. proof}, it follows from \cref{lem. integral identity} and \cref{lem.laplace squared} that
 \begin{align}
     & \left(\frac{t}{2}\right)^2
     \int_{M} v^s
     \left(|\nabla^2 v|^2
     +Ric(\nabla v,\nabla v)
     -v^{-1}|\nabla v|^2\Delta v
     \right)\notag\\
     =&\int_M |\nabla^2 u|^2
     +\left(\frac{2}{t}-1\right)^2\int_Mu^{-2}|\nabla u|^4
     +2\left(\frac{2}{t}-1\right)\int_M
     \left\langle\nabla^2 u,\frac{\dd u\otimes\dd u}{u}\right\rangle
     +\int_M Ric(\nabla u,\nabla u)\notag\\
     &-\frac{2}{t}\left(\frac{2}{t}-1\right)\int_M
     u^{-2}|\nabla u|^4
     -\frac{2}{t}\int_M
     u^{-1}|\nabla u|^2\Delta u\notag\\
     =&\int_M|\nabla^2 u|^2
     +\left(1-\frac{2}{t}\right)\int_M
     u^{-2}|\nabla u|^4
     +\int_M Ric(\nabla u,\nabla u)
     \notag\\
     &+2\left(\frac{2}{t}-1\right)\int_M
     \left\langle\nabla^2 u,\frac{\dd u\otimes\dd u}{u}\right\rangle
     -\frac{2}{t}\int_M
     u^{-1}|\nabla u|^2\Delta u
     \notag\\
     =&\int_M |\nabla^2 u|^2
     +\left(1-\frac{2}{t}+2\left(\frac{2}{t}-1\right)\frac{1}{n+2}
     -\frac{2}{t}\frac{n}{n+2}\right)\int_M
     u^{-2}|\nabla u|^4
     +\int_M Ric(\nabla u,\nabla u)\notag\\
     &+\left(
     2\left(\frac{2}{t}-1\right)\frac{n}{n+2}
     +\frac{2}{t}\frac{2n}{n+2}
     \right)\int_M
     \left\langle\nabla^2 u-\frac{\Delta u}{n}g,\frac{\dd u\otimes \dd u}{u}-\frac{1}{n}\frac{|\nabla u|^2}{u}g\right\rangle
     \notag\\
     =&\int_M |\nabla^2 u|^2
     +\frac{2n}{n+2}\left(\frac{4}{t}-1\right)\int_M
     \left\langle\nabla^2 u-\frac{\Delta u}{n}g,\frac{\dd u\otimes \dd u}{u}-\frac{1}{n}\frac{|\nabla u|^2}{u}g\right\rangle
     +\int_M Ric(\nabla u,\nabla u)
     \notag\\
     &+\frac{n}{n+2}\left(1-\frac{4}{t}\right)\int_M
     u^{-2}|\nabla u|^4
     \notag\\
     =&\int_M \alpha\left|\nabla^2 u-\frac{\Delta u}{n}g\right|^2
     +\frac{\alpha}{n}\int_M
     (\Delta u)^2
     +(1-\alpha)\int_M
     (\Delta u)^2
     +\alpha \int_M Ric(\nabla u,\nabla u)\notag\\
     &+\frac{2n}{n+2}\left(\frac{4}{t}-1\right)\int_M
     \left\langle\nabla^2 u-\frac{\Delta u}{n}g,\frac{\dd u\otimes \dd u}{u}-\frac{1}{n}\frac{|\nabla u|^2}{u}g\right\rangle
     +\frac{n}{n+2}\left(1-\frac{4}{t}\right)
     \int_M u^{-2}|\nabla u|^4\notag\\
     =&\int_M \left|\sqrt{\alpha}
     \left(\nabla^2 u-\frac{\Delta u}{n}g\right)
     +\frac{n}{n+2}\left(\frac{4}{t}-1\right)\frac{1}{\sqrt{\alpha}}
     \left(\frac{\dd u\otimes \dd u}{u}-\frac{1}{n}\frac{|\nabla u|^2}{u}g\right)\right|^2
     \notag\\
     &+\left(
     \frac{n}{n+2}\left(1-\frac{4}{t}\right)
     -\frac{n(n-1)}{(n+2)^2}\left(\frac{4}{t}-1\right)^2\frac{1}{\alpha}
     \right)
     \int_M u^{-2}|\nabla u|^4\notag
     \\
     &+\left(1-\frac{n-1}{n}\alpha\right)
     \int_M (\Delta u)^2
     +\alpha \int_M Ric(\nabla u,\nabla u).\label{eq. proof2}
 \end{align}
 Now we choose $\alpha$ such that the coefficient of $\int_M u^{-2}|\nabla u|^4$ in \cref{eq. proof2} vanishes, i.e.
 \begin{align*}
     \alpha=\frac{n-1}{n+2}\left(1-\frac{4}{t}\right)
     =\frac{(n-1)(s-2)}{(n+2)(s+2)}>0.
 \end{align*}
 Moreover, we want  the coefficients of the last two terms in \cref{eq. proof2} to be nonnegative. Notice that
 \begin{align*}
     \alpha>0&\Longleftrightarrow
     s\in (-\infty,-2)\cup \left(2,+\infty\right),\\
     \alpha\leq \frac{n}{n-1}
     &\Longleftrightarrow
     s\in\left(-\infty,-\frac{2(2n^2+1)}{4n-1}\right]\cup
     \left[-2,+\infty\right).
 \end{align*}
 Therefore, with $s\in\left(-\infty,-\frac{2(2n^2+1)}{4n-1}\right]\cup(2,+\infty)$, we could continue \cref{eq. proof2} as follows:
 \begin{align*}
     & \left(\frac{t}{2}\right)^2
     \int_{M} v^s
     \left(|\nabla^2 v|^2
     +Ric(\nabla v,\nabla v)
     -v^{-1}|\nabla v|^2\Delta v
     \right)\\
     =&\int_M \left|\sqrt{\alpha}
     \left(\nabla^2 u-\frac{\Delta u}{n}g\right)
     +\frac{n}{n+2}\left(\frac{4}{t}-1\right)\frac{1}{\sqrt{\alpha}}
     \left(\frac{\dd u\otimes \dd u}{u}-\frac{1}{n}\frac{|\nabla u|^2}{u}g\right)\right|^2
     \notag\\
     &+\left(
     \frac{n}{n+2}\left(1-\frac{4}{t}\right)
     -\frac{n(n-1)}{(n+2)^2}\left(\frac{4}{t}-1\right)^2\frac{1}{\alpha}
     \right)
     \int_M u^{-2}|\nabla u|^4\notag
     \\
     &+\left(1-\frac{n-1}{n}\alpha\right)
     \int_M (\Delta u)^2
     +\alpha \int_M Ric(\nabla u,\nabla u)\\
     \geq& \int_M \left(1-\frac{n-1}{n}\alpha\right)\lambda_1|\nabla u|^2
     +(n-1)\alpha\int_M |\nabla u|^2\\
     =&\left(\left(1-\frac{(n-1)^2(s-2)}{n(n+2)(s+2)}
        \right)\lambda_1
        +\frac{(n-1)^2(s-2)}{(n+2)(s+2)}
        \right)
        \left(\frac{t}{2}\right)^2
        \int_M v^s|\nabla v|^2.
 \end{align*}
This yields the desired inequality.
 The case for $s=2$ could be obtained by taking the limit $s\to 2^+$ in \cref{ineq. interpolated BE Gamma2}. 
\end{proof}
\begin{remark}\label{rmk. counterexample}
     For $s\in \left(-\frac{2(2n^2+1)}{4n-1},-2\right)$, there exists some function $0<v\in C^{\infty}(\mathbb{S}^n)$ such that
    \cref{ineq. interpolated BE Gamma2} fails to hold.
    We give a brief explanation as follows.
    
    For $s\in \left(-\frac{2(2n^2+1)}{4n-1},-2\right)$, we have $t:=s+2\in \left(\frac{-4(n-1)^2}{4n-1},0\right)$ and $\alpha:=\frac{n-1}{n+2}\left(1-\frac{4}{t}\right)>\frac{n}{n-1}$. Consider $u(x):=(2+\cos r(x))^{1-n}$ and $v:=u^{\frac{2}{t}}$. Then it is straightforward to verify that the first square term in the right hand side of \cref{eq. proof2} vanishes for this $u$, and 
    \begin{align*}
        \int_{\mathbb{S}^n}
     \left(1-\frac{n-1}{n}\alpha\right)(\Delta u)^2
     <\left(1-\frac{n-1}{n}\alpha\right)\lambda_1(\mathbb{S}^n)
     \int_{\mathbb{S}^n}|\nabla u|^2.
    \end{align*}
    Inserting these into \cref{eq. proof2} shows that $v$ fails to satisfy \cref{ineq. interpolated BE Gamma2}.

    However, for $s\in[-2,2)$, it is unknown whether \cref{ineq. interpolated BE Gamma2} holds even for the model space $\mathbb{S}^n$.
 \end{remark}
\section{Monotonicity of the Tsallis entropy }\label{sec. corollaries}
Consider the Tsallis entropy defined for $p\in(0,1)\cup(1,+\infty)$ and $0<u\in C^{\infty}(M)$,
\begin{align}\label{eq. Tsallis entropy}
    T_p(u):=\frac{\int_M u^p-\int_M u}{1-p}.
\end{align}
To get \cref{thm. ODE of T entropy}, we just need to calculate the expressions for the derivatives of the Tsallis entropy under the heat flow.

\begin{lemma}\label{lem. evolution of U}
Let $(M^n,g)$ be a compact manifold.
    Consider the heat flow $\partial_t u=\Delta u$ with the initial value $u_0>0$ and let $U:=u^{\frac{p-1}{2}}$. There hold
    \begin{align*}
    \nabla u=&\frac{2}{p-1}uU^{-1}\nabla U,\\
        \partial_t U
        =&\Delta U
        +\left(\frac{2}{p-1}-1\right)U^{-1}|\nabla U|^2.
    \end{align*}
\end{lemma}
\begin{proof}
Note that $\log U=\frac{p-1}{2}\log u$, we get
    \begin{align*}
        U^{-1}\nabla U=\frac{p-1}{2}u^{-1}\nabla u.
    \end{align*}
    Therefore we could express $\Delta u$ as
    \begin{align*}
        \Delta u
        =&\frac{2}{p-1}div(uU^{-1}\nabla U)\\
        =&\frac{2}{p-1}
        (uU^{-1}\Delta U
        +\langle\nabla u,\nabla U\rangle U^{-1}
        -uU^{-2}|\nabla U|^2
        )\\
        =&\frac{2}{p-1}uU^{-1}
        \left(
        \Delta U
        +\left(\frac{2}{p-1}-1\right)U^{-1}|\nabla U|^2
        \right).
    \end{align*}
    Hence we get 
    \begin{align*}
        \partial_t U
        =\frac{p-1}{2}Uu^{-1}\Delta u
        =\Delta U
        +\left(\frac{2}{p-1}-1\right)U^{-1}|\nabla U|^2.
    \end{align*}
\end{proof}

\begin{proposition}\label{prop. two derivatives of Tsallis entropy}
    Under the same assumptions and notations in \cref{lem. evolution of U}, we have
    \begin{align*}
        \frac{\dd}{\dd t}T_p(u)
        =&p\int_M u^{p-2}|\nabla u|^2
        =p\left(\frac{2}{p-1}\right)^2\int_M u|\nabla U|^2,\\
        \frac{\dd^2}{\dd t^2}T_p(u)
        =&-2p\left(\frac{2}{p-1}\right)^2
        \int_M u
        \left(
        |\nabla^2 U|^2+Ric(\nabla U,\nabla U)
        -U^{-1}|\nabla U|^2\Delta U
        \right).
    \end{align*}
\end{proposition}
\begin{proof}
First, the derivative of the $L^1$ norm of $u$ vanishes:
    \begin{align*}
        \partial_t\int_M u=\int_M \Delta u=0.
    \end{align*}
    For the $L^p$ norm term, by the heat equation and \cref{lem. evolution of U} we have
    \begin{align*}
        \partial_t\int_M u^p
        =&p\int_M u^{p-1}\Delta u
        =p(1-p)\int_M u^{p-2}|\nabla u|^2
        =p(1-p)
        \left(\frac{2}{p-1}\right)^2\int_M u^pU^{-2}|\nabla U|^2\\
        =&p(1-p)\left(\frac{2}{p-1}\right)^2\int_M u|\nabla U|^2.
    \end{align*}
    Therefore, we could get the first order derivative of $T_p(u)$.
    
    For the second order derivative of $T_p(u)$, by \cref{lem. evolution of U} we have
    \begin{align}\label{eq. intermediate step second deri}
        \partial_t\int_M u|\nabla U|^2
        =&\int_M |\nabla U|^2\Delta u
        +2\int_M u\langle\nabla \partial_tU,\nabla U\rangle\notag\\
        =&\int_M u\Delta |\nabla U|^2
        +2\int_M u\langle\nabla \Delta U,\nabla U\rangle\notag\\
        &+2\left(1-\frac{2}{p-1}\right)\int_M uU^{-2}|\nabla U|^4
        +2\left(\frac{2}{p-1}-1\right)\int_M uU^{-1}\langle\nabla|\nabla U|^2,\nabla U\rangle.
    \end{align}
    By \cref{lem. evolution of U} and the divergence theorem, we tackle the second term in the last line of \cref{eq. intermediate step second deri} as
    \begin{align}\label{eq. second ugly term}
        \int_M uU^{-1}\langle\nabla |\nabla U|^2,\nabla U\rangle
        =&\frac{p-1}{2}\int_M \langle\nabla|\nabla U|^2,\nabla u\rangle
        =-\frac{p-1}{2}\int_M u\Delta |\nabla U|^2
    \end{align}
    To tackle the first term in the last line of \cref{eq. intermediate step second deri},  we notice that by \cref{lem. evolution of U} and \cref{eq. second ugly term},
    \begin{align*}
        \int_M uU^{-1}|\nabla U|^2\Delta U
        =&-\int_M \langle\nabla u,\nabla U\rangle U^{-1}|\nabla U|^2
        +\int_M uU^{-2}|\nabla U|^4
        -\int_M uU^{-1}\langle\nabla |\nabla U|^2,\nabla U\rangle\\
        =&\left(1-\frac{2}{p-1}\right)\int_M uU^{-2}|\nabla U|^4
        +\frac{p-1}{2}\int_M u\Delta|\nabla U|^2.
    \end{align*}
    This is equivalent to
    \begin{align}\label{eq. first ugly term}
        2\left(1-\frac{2}{p-1}\right)\int_M uU^{-2}|\nabla U|^4
        =2\int_M uU^{-1}|\nabla U|^2\Delta U
        -(p-1)\int_M u\Delta |\nabla U|^2.
    \end{align}
    Inserting \cref{eq. first ugly term} and \cref{eq. second ugly term} into \cref{eq. intermediate step second deri}, we conclude by Bochnner formula that
    \begin{align*}
        \partial_t\int_M u|\nabla U|^2
        =&-\int_M u\Delta |\nabla U|^2
        +2\int_M u\langle\nabla \Delta U,\nabla U\rangle
        +2\int_M uU^{-1}|\nabla U|^2\Delta U\\
        =&-2\int_M u(|\nabla^2 U|^2+Ric(\nabla U,\nabla U))
        +2\int_M uU^{-1}|\nabla U|^2\Delta U.
    \end{align*}
    This gives the second order derivative of $T_p(u)$.
\end{proof}
Now \cref{thm. ODE of T entropy} would be a consequence of \cref{thm. modified weighted Gamma2} and 
\cref{prop. two derivatives of Tsallis entropy}.
\begin{proof}[Proof of \cref{thm. ODE of T entropy}:]
    For $p\in\left[\frac{2(n-1)^2}{2n^2+1},1\right)\cup(1,2]$, we take
    \begin{align*}
        s:=\frac{2}{p-1}\in\left(-\infty,-\frac{2(2n^2+1)}{4n-1}\right]\cup[2,+\infty).
    \end{align*}
    Use the notation $U:=u^{\frac{p-1}{2}}$ as in \cref{lem. evolution of U}. By \cref{prop. two derivatives of Tsallis entropy} and \cref{thm. modified weighted Gamma2}, we get
    \begin{align*}
        &-\frac{1}{2p}\left(\frac{p-1}{2}\right)^2
        \frac{\dd^2}{\dd t^2}T_p(u)
        =\int_M U^{\frac{2}{p-1}}(|\nabla^2 U|^2+Ric(\nabla U,\nabla U)-U^{-1}|\nabla U|^2\Delta U)\\
        \geq& 
        \left(\left(
        1-\frac{(n-1)^2}{n(n+2)}\left(\frac{2}{p}-1\right)
        \right)\lambda_1
        +\frac{(n-1)^2}{n+2}\left(\frac{2}{p}-1\right)
        \right)
        \int_M U^{\frac{2}{p-1}}|\nabla U|^2\\
        =&\left(\left(
        1-\frac{(n-1)^2}{n(n+2)}\left(\frac{2}{p}-1\right)
        \right)\lambda_1
        +\frac{(n-1)^2}{n+2}\left(\frac{2}{p}-1\right)
        \right)
        \frac{1}{p}
        \left(\frac{p-1}{2}\right)^2
        \frac{\dd}{\dd t}T_p(u).
    \end{align*}
\end{proof}
\begin{remark}\label{rmk. counterexample for ODE}
    For $p\in\left(0,\frac{2(n-1)^2}{2n^2+1}\right)$ and $(M^n,g)=\mathbb{S}^n$, there exists a positive function $u_0$ such that  \cref{ineq. in thm ODE of T entropy} fails for the heat flow with the initial value $u_0$  at $t=0$. In fact, this is a direct consequence of \cref{rmk. counterexample} by noting that $s:=\frac{2}{p-1}\in
        \left(-\frac{2(2n^2+1)}{4n-1},-2\right)$. 
    However, we do not know whether \cref{ineq. in thm ODE of T entropy} holds for $p< 0$ or $p>2$ even on the model space $\mathbb{S}^n$.
\end{remark}

\begin{proof}[Proof of \cref{cor. some Sobolev inequalities}:]
Let $p:=\frac{2}{q}\in\left[\frac{2(n-1)^2}{2n^2+1},1\right)\cup(1,2]$.
Given $0<v\in C^{\infty}(M)$, define $u_0:=v^q=v^{\frac{2}{p}}$. Consider the heat flow
    \begin{align*}
        \begin{cases}
            \partial_tu=\Delta u\\
            u(0,\cdot)=u_0.
        \end{cases}
    \end{align*}
    Note that
    \begin{align*}
        \lim_{t\to+\infty}u(t,\cdot)=\fint_M u_0.
    \end{align*}
    Then we use \cref{thm. ODE of T entropy} and integrate for $t$ over $[0,+\infty)$ and get
    \begin{align*}
        0
        \geq&0-p\int_M u_0^{p-2}|\nabla u_0|^2
        +2\left(\left(
        1-\frac{(n-1)^2}{n(n+2)}\left(\frac{2}{p}-1\right)
        \right)\lambda_1
        +\frac{(n-1)^2}{n+2}\left(\frac{2}{p}-1\right)
        \right)\times\\
        &\quad\quad\quad\quad\quad\quad\quad\quad\left(
        \frac{Vol(M)(\fint_Mu_0)^p-Vol(M)\fint_M u_0}{1-p}
        -\frac{\int_M u_0^p-\int_M u_0}{1-p}
        \right)
    \end{align*}
    Rearranging it, we get
    \begin{align*}
        &p\fint_M u_0^{p-2}|\nabla u_0|^2\\
        \geq&\frac{2}{1-p}
        \left(\left(
        1-\frac{(n-1)^2}{n(n+2)}\left(\frac{2}{p}-1\right)
        \right)\lambda_1
        +\frac{(n-1)^2}{n+2}\left(\frac{2}{p}-1\right)
        \right)
        \left(
        \left(\fint_M u_0\right)^p-\fint_M u_0^p
        \right).
    \end{align*}
    Finally, substitute $u_0=v^{\frac{2}{p}}$ and $p=\frac{2}{q}$ to conclude that
    \begin{align*}
        2q\fint_M|\nabla v|^2
        \geq
        \frac{2q}{q-2}
        \left(\left(
        1-\frac{(n-1)^2}{n(n+2)}(q-1)
        \right)\lambda_1
        +\frac{(n-1)^2}{n+2}(q-1)
        \right)
        \left(
        \left(\fint_M v^q\right)^{\frac{2}{q}}
        -\fint_M v^2
        \right).
    \end{align*}
\end{proof}
We conclude this paper with several remarks.

We note that the Sobolev constant in \cref{cor. some Sobolev inequalities} is sharp in the sense that it is the largest possible one for the model manifold $\mathbb{S}^n$. However, it is slightly weaker than the Sobolev constant in \cref{thm. DEL14} for general manifolds with $Ric\geq n-1$, since $\lambda_1\geq n$. The same phenomenon occurs for the logarithmic Sobolev inequality: the lower bound of $\alpha_M$ given by \cref{thm. Ji} together with the Bakry-\'Emery criterion \cref{ineq. Bakry Emery criterion} is slightly weaker than that obtained by Rothaus \cref{ineq. Rothaus}. Therefore, it is natural to wonder whether the lower bound on $\Lambda _M$
in \cref{thm. Ji} can be improved to match that in \cref{ineq. Rothaus}. Similarly, it would also be valuable to refine the coefficient in \cref{ineq. interpolated BE Gamma2} so that it yields a Sobolev inequality with the same Sobolev constant as in \cref{thm. DEL14}.

\cref{thm. ODE of T entropy} and \cref{cor. some Sobolev inequalities} unveil that the Tsallis entropy exhibits certain monotonicity under the heat flow, leading to sharp Sobolev inequalities with exponent $q\in [1,2)\cup\left(2,\frac{2n^2+1}{(n-1)^2}\right]$. 
However, as we mentioned in \cref{rmk. counterexample for ODE}, \cref{thm. ODE of T entropy} fails for $p\in(0,\frac{2(n-1)^2}{2n^2+1})$.
It would be of interest to determine whether other monotonicity properties of Tsallis entropy can be established so that the Sobolev inequality for the full range of the Sobolev exponent $q\in (2,\frac{2n}{n-2}]$ could be recovered. 

Finally, we notice that \cref{ineq. in thm ODE of T entropy} is equivalent to
\begin{align*}
    \frac{\dd}{\dd t}\left(
    C(n,p)t+
    \log \left(
    \frac{\dd}{\dd t}T_p(u)
    \right)
    \right)\leq 0,
\end{align*}
where $C(n,p):=2\left(\left(
        1-\frac{(n-1)^2}{n(n+2)}\left(\frac{2}{p}-1\right)
        \right)\lambda_1
        +\frac{(n-1)^2}{n+2}\left(\frac{2}{p}-1\right)
        \right)$. This combined with \cref{prop. two derivatives of Tsallis entropy} yields a decay rate upper bound of a weighted Dirichlet energy under the heat flow:
        \begin{align*}
            \int_M u|\nabla U|^2\leq e^{-C(n,p)t}\int_M u_0|\nabla U_0|^2,\quad \forall t\geq 0,
        \end{align*}
        where $\partial_tu=\Delta u$ is the heat flow with a positive initial data $u_0$, and $U:=u^{\frac{p-1}{2}}$.
        On the other hand, by considering the $L^2$ spectral decomposition of the solution $u$ and the fact that $u$ converges to a positive constant function as $t\to+\infty$, we know that  the decay rate has an asymptotic  lower bound in terms of $\lambda_1$: 
        \begin{align*}
            Ce^{-\lambda_1t}\leq\int_M u|\nabla U|^2,\quad t\gg 1,
        \end{align*}
        where $C$ is a positive constant depending only on $u_0$.
        This viewpoint indicates some restrictions on the constant appearing in  \cref{ineq. interpolated BE Gamma2} and 
        \cref{ineq. in thm ODE of T entropy}.

\bibliographystyle{amsplain}
\bibliography{references}
\end{document}